\documentclass[12pt]{amsart}
\usepackage[margin=1.in, centering]{geometry}

\usepackage{amssymb,latexsym,amsmath,extarrows, mathrsfs, amsthm}
\usepackage{graphicx}

\usepackage[colorlinks, linkcolor=blue]{hyperref}

\usepackage{multirow}
\usepackage{color}
\usepackage{wasysym}

\usepackage{float}

\newtheorem{theorem}{Theorem}[section]
\newtheorem{lemma}[theorem]{Lemma}
\newtheorem{proposition}[theorem]{Proposition}

\newtheorem{definition}[theorem]{Definition}

\newtheorem{corollary}[theorem]{Corollary}

\newtheorem{conjecture}[theorem]{Conjecture}

\theoremstyle{remark}
\newtheorem*{remark}{\bf  \itshape  \textup{Remark}}

\allowdisplaybreaks
\title{Visible lattice points in  P\'{o}lya's walk}

\author{Meijie Lu}
\address{School of Mathematics, Shandong University, Jinan 250100, Shandong, China}
\email{meijie.lu@hotmail.com}

\author{Xianchang Meng}
\address{School of Mathematics, Shandong University, Jinan 250100, Shandong, China \& Johannes Kepler University Linz, Linz 4040, Austria}
\email{xianchang.meng@jku.at}


\begin{document}
	
	
	\maketitle
	\begin{abstract}	
		In this paper, for any integer $k\geq 2$, we  study the distribution of the visible lattice points in  certain generalized P\'{o}lya's walk on $\mathbb{Z}^k$: perturbed P\'{o}lya's walk and twisted P\'{o}lya's walk.
		For the first case, we prove that the density of visible lattice points in a perturbed P\'{o}lya's walk  is almost surely $1/\zeta(k)$, where $\zeta(s)$ denotes the Riemann zeta function. A trivial case of our result covers the standard P\'{o}lya's walk. Moreover, we do numerical experiments for the second case, we conjecture that the density is also almost surely $1/\zeta(k)$. 
	\end{abstract}
	
	\section{Introduction}
	
	\subsection{Visible lattice points on ${\mathbb{Z}^{k}}$}
	Let $k\geq 2$ be an integer, in the $k$-dimensional integer lattice ${\mathbb{Z}^{k}}$, a lattice point ${\bf n}=(n_1,\cdots,n_k)$ is said to be visible  from ${\bf m}=(m_1,\cdots,m_k)$ if the straight line segment joining ${\bf n}$ and ${\bf m}$ contains no other lattice point. If ${\bf n}$ is visible from the origin, we simply call it a visible point. This concept has been studied and used in various areas such as  integer optimization and 
	theoretical physics \cite{BGW,BCZ,M}.
	
	Visible lattice points are also widely concerned in number theory, for example, their density (or their numbers) in certain regions. 
	Dirichlet \cite{D-1} and Sylvester \cite{S} proved that the density of visible lattice points on $\mathbb{Z}^2$ is $1/\zeta(2)=6/\pi^2\approx0.607928$, where $\zeta(s)$ stands for the Riemann zeta function. Lehmer \cite{L} and Christopher \cite{C} considered the visible lattice points in higher dimensional lattices. They showed that the proportion of such points on $\mathbb{Z}^k$ is $1/\zeta(k)$. Moreover, in  \cite{B,HN,Z}, the authors gave asymptotic formulas for the number  of visible points within a suitable smooth convex domains in $\mathbb{R}^2$.
	
	All the above results focus on the lattice points which are visible to only one point (i.e. the origin). Rearick in his Ph.D. thesis \cite{R-thesis} first considered the lattice points which are visible from multiple $M (M=2\ {\rm or}\ 3)$ points, he proved that the density  of such points on $\mathbb{Z}^2$ is $\prod_{p}(1-M/p^2)$, where $p$ runs over all primes. Then in \cite{R}, he generalized this result to the higher dimensional lattice $\mathbb{Z}^k$, as well as a general $M$, and gave an asymptotic formula for the number of such visible points. Later,  Liu, Lu and Meng \cite{LLM} improved the error term of this asymptotic formula.  Further works on simultaneous visibility can be found in \cite{BKP,CTZ,R(1)}.

	Other interesting questions related to visible points have been investigated before. For example, the minimum number problem, that is, the least number of lattice points that can be selected from $[1,n]^k\cap \mathbb{Z}^k$ such that each point of this lattice cube is
	visible from at least one of them. The reader can refer to  \cite{Ab,AB,AC} for the bound of this minimum number. Also, the visible points can be considered from its diffraction pattern \cite{BMP}, its ergodic properties \cite{BH}, or its percolative properties \cite{M(1)}.
	
	The above concept of visibility  can be generalized, namely, to consider the visibility along certain type of curves. See some recent works in this direction in \cite{BEH, GHKM, HO, LM-2, LLM}.

	\subsection{P\'{o}lya's walk  on ${\mathbb{Z}^{k}}$}
	Random walks have received extensive attention in various disciplines: in mathematics, in biology and even in chemistry \cite{B(1),KPRT}.
	In number theory, Lifshits and Weber \cite{LW} and Srichan \cite{Sr} considered the Lindel\"{o}f hypothesis in Cauchy random walk; 
	Jouve \cite{J} connected the large sieve method with random walks on cosets of arithmetic group; McNew \cite{Mc} studied
	random walks on the residues modulo $n$. 
	In this paper, we focus on another type of random walks which is the so-called  P\'{o}lya's walk.
	
	Let $\mathbb{Z}_{>0}$ be the set of positive integers, as Hales mentioned in \cite{H}, we describe the standard P\'{o}lya's walk $({\bf p}_{i})_{i\geq 0}$ on $\mathbb{Z}^2$  as follows: the walker starts from the initial point ${\bf p}_0=(p_{0,1},p_{0,2})\in \mathbb{Z}^2_{>0}$, and the jumps of the walker ${\bf p}_{i+1}-{\bf p}_i$ can only take two possible values: (1,0) and (0,1) (that is, the walker moves one unit right or up in $\mathbb{Z}^2$ at each step), where ${\bf p}_i=(p_{i,1},p_{i,2})$ is the coordinate of the $i$th step of the walker, and the corresponding probabilities are 
	$$
	\mathbb{P}\big({\bf p}_{i+1}-{\bf p}_i=(1,0)\big)=\frac{p_{i,1}}{p_{i,1}+p_{i,2}},\ \ \ \ \ \mathbb{P}\big({\bf p}_{i+1}-{\bf p}_i=(0,1)\big)=\frac{p_{i,2}}{p_{i,1}+p_{i,2}}.
	$$ 
	It is natural to generalize this concept to higher dimensional lattice. Here we are interested in a general P\'{o}lya's walk on $\mathbb{Z}^{k}$. Let $s({\bf  v}):=v_1+\cdots+v_k$ for vector ${\bf v}=(v_1,\cdots,v_k)\in\mathbb{Z}^k$, we define the perturbed P\'{o}lya's walk as following:
	
	\begin{definition}[Perturbed P\'{o}lya's walk]\label{Defn-Polya-walk}
		Suppose ${\bf p}_0=(p_{0,1},\cdots,p_{0,k})\in\mathbb{Z}^{k}$ with all $p_{0,r}\geq B$, here $B>0$ is a fixed constant, for fixed vector $\boldsymbol{\beta}=(\beta_1,\cdots,\beta_k)$ satisfying all $|\beta_{r}|<B$ and $\beta_1+\cdots+\beta_k=0$, a perturbed P\'{o}lya walk $({\bf p}_{i})_{i\geq 0}$ started from ${\bf p}_0$ is given by
		$$
		{\bf p}_{i+1}={\bf p}_{i}+{\bf Y}_{i+1}
		$$
		for $i=0,\ 1,\cdots$, where ${\bf p}_i=(p_{i,1},\cdots,p_{i,k})$ is the coordinate of the $i$th step and the jumps
		\begin{align}\label{def:Y=}
			{\bf Y}_{i+1}:=\left\{
			\begin{aligned}
				&(1,0,\cdots,0),\ \ {{with\ probability}}\ \frac{p_{i,1}+\beta_1}{s({\bf p}_{i})},\\
				&(0,1,\cdots,0),\ \ {{with\ probability}}\ \frac{p_{i,2}+\beta_2}{s({\bf p}_{i})},\\
				&\quad\quad\quad\quad\cdots\\
				&(0,0,\cdots,1),\ \ {{with\ probability}}\ \frac{p_{i,k}+\beta_k}{s({\bf p}_{i})}.\\
			\end{aligned}
			\right.
		\end{align}
	\end{definition}
	We emphasize that the above jumps are not independent. Besides, if the constant $B=1$, and there is no perturbation, i.e. all $\beta_r=0$ in \eqref{def:Y=}, then
	the corresponding walk is the  standard P\'{o}lya's walk with initial position ${\bf p}_0\in\mathbb{Z}^k_{>0}$. In fact, the  standard P\'{o}lya's walk can be interpreted in a different way, i.e. the P\'{o}lya's urn model:
	
	Suppose there is an urn that consists of $k$ different colored balls, and for convenience, we number these colors as $1,2,\cdots,k$. At the begining, suppose there are $p_{0,r}$ balls of color $r,\ r=1,\ 2,\cdots,k$. Drawing a ball uniformly at random from the urn at each time, and observe the color, then return the drawn ball to the urn and add a ball of the same color. Repeat this process indefinitely. Let $p_{n,r}$ be the number of balls of color $r$ after $n$ successive random drawings, then the list $(p_{n,1},\cdots,p_{n,k})$ denotes the composition of P\'{o}lya's urn after $n$ drawings, and it is also the coordinate of the $n$th step ${\bf p}_n$ of the standard P\'{o}lya's walk. For more information about this model, see \cite{Ma}.
	
	\subsection{Visible lattice points in perturbed P\'{o}lya's walk}
	Associated to the perturbed P\'{o}lya's walk on $\mathbb{Z}^k$, we consider a sequence of  Bernoulli random variables $(V_i)_{i\geq1}$ with each
	\begin{equation}\label{def:V=}
		V_i:=\left\{
		\begin{aligned}
			&1,\ \ {\rm{if}}\ {\bf p}_i\ \rm{is\ visible},\\
			&0,\ \ \rm{otherwise.}
		\end{aligned}
		\right.
	\end{equation}
	For integer $N\geq 1$, we denote
	\begin{equation}\label{def:R=}
		\overline{R}_k({N}):=\overline{R}_k(N;\boldsymbol{\beta},{\bf p}_0)=\frac{1}{N}(V_1+\cdots+V_N),
	\end{equation}
	then the random variable $\overline{R}_k({N})$ represents the proportion of visible steps in the perturbed P\'{o}lya's walk in first $N$ steps. 
	
	Under the above notations, the first result of this paper is
	\begin{theorem}\label{thm:R(S)=}
		Suppose ${\bf p}_0 $ and $\boldsymbol{\beta}$ are given as in Definition \ref{Defn-Polya-walk}, then for any integer $k\geq 2$, we have
		$$
		\lim_{N\rightarrow\infty}\overline{R}_k({N})=\frac{1}{\zeta(k)}\ \ \ almost\ surely,
		$$
		where $\zeta(s)$ is the Riemann zeta function.
	\end{theorem}
	
	Theorem \ref{thm:R(S)=} tells us  that the density of visible lattice points on a random path of the above perturbed P\'{o}lya's walk on $\mathbb{Z}^k$ is almost surely $1/\zeta(k)$. In other words,  this density in such a random walk is almost surely the asymptotic density of visible points in the whole space.
	
	If we assume $B=1$ and $\beta_r=0$ for all $1\leq r\leq k$ in the definition of perturbed P\'{o}lya's walk, then Theorem \ref{thm:R(S)=} holds for the standard P\'{o}lya's walk, we give the following corollary.
	\begin{corollary}\label{cor:standard density}
		The density of visible lattice points visited by a $k$-dimensional standard P\'{o}lya's walk with starting point ${\bf p}_0\in\mathbb{Z}^k_{>0}$ is almost surely
		${1}/{\zeta(k)}$.
	\end{corollary}
	
	Observe that  Corollary \ref{cor:standard density} covers Theorem 1 of \cite{FF} for $k=2$. 
	
	Next we verify the above results by doing experiments using random generator in Python. In this process, we consider the perturbed P\'{o}lya's walk in two or three-dimensional lattice, and calculate the density of visible steps within 100000 steps. Since our results are ``almost surely", we do the same calculation 10 times then take the average. We list the numerical densities on $\mathbb{Z}^2$ and $\mathbb{Z}^3$ for some values of ${\bf p}_0$, $\boldsymbol{\beta}$ and $B$ in Tables \ref{tab:the generalized polya walk in two} and \ref{tab:the generalized polya walk in three}, respectively. 
	
	We should point out that the densities for $\boldsymbol{\beta}=(0,0)$ and $\boldsymbol{\beta}=(0,0,0)$ in Tables \ref{tab:the generalized polya walk in two} and \ref{tab:the generalized polya walk in three} are the numerical densities in standard P\'{o}lya's walk started at ${\bf p}_0$. 
	
	Recall from Theorem \ref{thm:R(S)=} that the theoretical density on $\mathbb{Z}^2$ (or $\mathbb{Z}^3$) is almost surely ${1}/{\zeta(2)}\approx0.607928$ (or ${1}/{\zeta(3)}\approx0.831907$). We see that,  the numerical densities in the following two tables fit very well with the conclusion in our Theorem \ref{thm:R(S)=}.

	\begin{table}[h]
		
		\begin{tabular}{|c|c|c|c|}
			\hline
			&  ${\bf p}_0$ & ${\boldsymbol{\beta}}$ & Numerical density\\
			\hline
			\multirow{4}*{$B=1$} & \multirow{2}*{$(1,1)$} & $(0,0)$ & 0.607958\\
			\cline{3-4}
			&  & $(0.6,-0.6)$ & 0.605907\\
			\cline{2-4}
			& \multirow{2}*{$(2,4)$} & $(-0.3,0.3)$ & 0.607895\\
			\cline{3-4}
			& & $(-0.8,0.8)$ & 0.607566\\
			\hline
			\multirow{4}*{$B=6.5$} & \multirow{2}*{$(7,9)$} & $(2.5,-2.5)$ & 0.608319\\
			\cline{3-4}
			&  & $(6,-6)$ & 0.608122\\
			\cline{2-4}
			& \multirow{2}*{$(10,8)$} & $(-3.5,3.5)$ & 0.608336\\
			\cline{3-4}
			& & $(-5,5)$ & 0.607564\\
			\hline
			\multirow{4}*{$B=10$} & \multirow{2}*{$(11,14)$} & $(7,-7)$ & 0.607236\\
			\cline{3-4}
			&  & $(9.5,-9.5)$ & 0.607387\\
			\cline{2-4}
			& \multirow{2}*{$(12,10)$} & $(-7.5,7.5)$ & 0.608467\\
			\cline{3-4}
			& & $(-8,8)$ & 0.606947\\
			\hline
			\multirow{4}*{$B=15.5$} & \multirow{2}*{$(16,18)$} & $(11.5,-11.5)$ & 0.608013\\
			\cline{3-4}
			&  & $(14,-14)$ & 0.607194\\
			\cline{2-4}
			& \multirow{2}*{$(20,17)$} & $(-12,12)$ & 0.607943\\
			\cline{3-4}
			& & $(-13.5,13.5)$ & 0.607483\\
			\hline
		\end{tabular}
		\caption{Density of visible  points in perturbed P\'{o}lya's walk  on  $\mathbb{Z}^2$}\label{tab:the generalized polya walk in two}
	\end{table}

	\begin{table}[h]
		\begin{tabular}{|c|c|c|c|}
			\hline
			&  ${\bf p}_0$ & ${\boldsymbol{\beta}}$ & Numerical density\\
			\hline
			\multirow{4}*{$B=1$} & \multirow{2}*{$(1,1,1)$} & $(0,0,0)$ & 0.831783\\
			\cline{3-4}
			&  & $(0.2,-0.3,0.1)$ & 0.831642\\
			\cline{2-4}
			& \multirow{2}*{$(2,3,6)$} & $(-0.3,0.5,-0.2)$ & 0.831845\\
			\cline{3-4}
			& & $(-0.7,-0.1,0.8)$ & 0.831466\\
			\hline
			\multirow{4}*{$B=6.5$} & \multirow{2}*{$(7,9,7)$} & $(2.5,2,-4.5)$ & 0.832073\\
			\cline{3-4}
			&  & $(3,-4.5,1.5)$ & 0.832442\\
			\cline{2-4}
			& \multirow{2}*{$(10,8,9)$} & $(-2.5,4.5,-2)$ & 0.832214\\
			\cline{3-4}
			& & $(-1,5.5,-4.5)$ & 0.831663\\
			\hline
			\multirow{4}*{$B=10$} & \multirow{2}*{$(11,14,15)$} & $(7,-3.5,-3.5)$ & 0.831966\\
			\cline{3-4}
			&  & $(8,1.5,-9.5)$ & 0.831787\\
			\cline{2-4}
			& \multirow{2}*{$(12,10,10)$} & $(-7.5,7.5,0)$ & 0.831789\\
			\cline{3-4}
			& & $(-8,-1.5,9.5)$ & 0.832312\\
			\hline
			\multirow{4}*{$B=15.5$} & \multirow{2}*{$(16,18,19)$} & $(11.5,-2,-9.5)$ & 0.831314\\
			\cline{3-4}
			&  & $(14,-8.5,-5.5)$ & 0.831345\\
			\cline{2-4}
			& \multirow{2}*{$(20,17,18)$} & $(-12,11.5,0.5)$ & 0.832540\\
			\cline{3-4}
			& & $(-13.5,-0.5,14)$ & 0.832322\\
			\hline
		\end{tabular}
		\caption{Density of visible  points in perturbed P\'{o}lya's walk on  $\mathbb{Z}^3$}\label{tab:the generalized polya walk in three}
	\end{table}

	\subsection{Visible lattice  points in twisted P\'{o}lya's walk}
	We consider another generalization of the Polya's walk in this subsection. 
	\begin{definition}[Twisted P\'{o}lya's walk]\label{def:twisted polya walk}
		Let  $\boldsymbol{\gamma}_r=(\gamma_{r,1},\cdots,\gamma_{r,k})$, $\forall 1\leq r\leq k$ be fixed non-zero vectors such that $0\leq\gamma_{r,1},\cdots,\gamma_{r,k}\leq1$ for each $r$ and $\boldsymbol{\gamma}_1+\cdots+\boldsymbol{\gamma}_k={\bf 1}=(1,\cdots,1)$, we define the twisted P\'{o}lya's walk $({\bf q}_{i})_{i\geq 0}$ with the starting point ${\bf q}_0=(q_{0,1},\cdots,q_{0,k})\in\mathbb{Z}^k_{>0}$:
		$$
		{\bf q}_{i+1}={\bf q}_{i}+{\bf Z}_{i+1}
		$$
		for $i=0,\ 1,\cdots$, where the random vectors
		\begin{align*}
			{\bf Z}_{i+1}:=\left\{
			\begin{aligned}
				&(1,0,\cdots,0),\ \ {{with\ probability}}\ \frac{\boldsymbol{\gamma}_1\cdot {\bf q}_{i}}{s({\bf q}_{i})},\\
				&(0,1,\cdots,0),\ \ {{with\ probability}}\ \frac{\boldsymbol{\gamma}_2\cdot {\bf q}_{i}}{s({\bf q}_{i})},\\
				&\quad\quad\quad\quad\cdots\\
				&(0,0,\cdots,1),\ \ {{with\ probability}}\ \frac{\boldsymbol{\gamma}_k\cdot {\bf q}_{i}}{s({\bf q}_{i})}.\\
			\end{aligned}
			\right.
		\end{align*}
		with $\boldsymbol{\gamma}\cdot{\bf q}$ denotes the inner product of the two vectors.
	\end{definition}
	For the special case $\boldsymbol{\gamma}_r={\bf e}_{r}$ for all $1\leq r\leq k$ (here ${\bf e}_{r}\in\mathbb{Z}^k$ stands for the vector whose $r$th component is 1 and 0 elsewhere), the above twisted walk gives the standard P\'{o}lya's walk on $\mathbb{Z}^k$. 
	
	For the twisted P\'{o}lya's walk, we are not equipped with  tools to theoretically derive the density of visible lattice points except for the special case when $\boldsymbol{\gamma}_r={\bf e}_{r}$. This is because in the proof of Theorem \ref{thm:R(S)=}, we take advantage of the fact that the joint mass  function of the jumps ${\bf Y}_i$ only depend on the number of ${\bf e}_1,\cdots,{\bf e}_k$ appearing among ${\bf Y}_1,\cdots,{\bf Y}_n$ but not on the order of ${\bf Y}_i$ (see Lemma \ref{lem:the joint mass function}), while in the generic  twisted P\'{o}lya's walk, the joint mass function of these ${\bf Z}_i$ does not have such property.
	Nevertheless, we do some experiments by Python to numerically predict the value of this density in the twisted walk. Here we also focus on the case $k=2$ or 3 and take some values of $\boldsymbol{\gamma}_r$ and ${\bf q}_0$, we give the numerical densities in Tables \ref{tab:twisted walk in two} and \ref{tab:twisted walk in three}, respectively.
	
	\begin{table}[h]
		\begin{tabular}{|c|c|c|c|}
			\hline
			$\boldsymbol{\gamma}_1$ & $\boldsymbol{\gamma}_2$ &Numerical for ${\bf q}_0=(1,1)$& Numerical for ${\bf q}_0=(2,4)$\\
			\hline
			(0,\ 1) & (1,\ 0) & 0.608150&0.608135\\
			\hline
			(0.1,\ 0.2) & (0.9,\ 0.8) & 0.607663& 0.607611\\
			\hline
			(0.1,\ 0.35) & (0.9,\ 0.65) & 0.608554& 0.608025\\
			\hline
			(0.3,\ 0.45) & (0.7,\ 0.55) & 0.607971& 0.607023\\
			\hline
			(0.3,\ 0.6) & (0.7,\ 0.4) & 0.607800& 0.608095\\
			\hline
			(0.6,\ 0.2) & (0.4,\ 0.8) &0.607967& 0.607320\\
			\hline
			(0.6,\ 0.6) & (0.4,\ 0.4) & 0.607713& 0.607452\\
			\hline
			(0.9,\ 0.75) & (0.1,\ 0.25) & 0.607861& 0.608326\\
			\hline
			(0.9,\ 0.9) & (0.1,\ 0.1)&0.607695& 0.607914\\
			\hline
			(1,\ 0) & (0,\ 1) & 0.608398& 0.608098\\
			\hline
		\end{tabular}
		\caption{Density of visible  points in twisted P\'{o}lya's walk on   $\mathbb{Z}^2$}\label{tab:twisted walk in two}
	\end{table}
	
	\begin{table}[h]
		\begin{tabular}{|c|c|c|c|c|}
			\hline
			$\boldsymbol{\gamma}_1$ & $\boldsymbol{\gamma}_2$& $\boldsymbol{\gamma}_3$&  ${\bf q}_0=(1,1,1)$& ${\bf q}_0=(2,3,6)$\\
			\hline
			(0,\ 1,\ 0) & (0,\ 0,\ 1)& (1,\ 0,\ 0) & 0.832074&0.832303\\
			\hline
			(0,\ 0,\ 1) & (1,\ 0,\ 0)& (0,\ 1,\ 0) & 0.831954&0.832463\\
			\hline
			(0.1,\ 0.2,\ 0.3) & (0.1,\ 0.3,\ 0.4) & (0.8,\ 0.5,\ 0.3) &0.831710&0.832078\\
			\hline
			(0.2,\ 0.3,\ 0.4) & (0.2,\ 0.5,\ 0.6)& (0.6,\ 0.2,\ 0)  & 0.831557&0.831922\\
			\hline
			(0.3,\ 0.1,\ 0.6) & (0.3,\ 0.4,\ 0.2)& (0.4,\ 0.5,\ 0.2) & 0.831985&0.831551\\
			\hline
			(0.5,\ 0.1,\ 0.4) & (0.5,\ 0.6,\ 0.5)& (0,\ 0.3,\ 0.1)& 0.832265&0.832151\\
			\hline
			(0.7,\ 0.1,\ 0.1) & (0.1,\ 0.9,\ 0.2) & (0.2,\ 0,\ 0.7) &0.831360&0.832062\\
			\hline
			(0.8,\ 0.2,\ 0.1) & (0.1,\ 0.7,\ 0.1)& (0.1,\ 0.1,\ 0.8) & 0.831532&0.832084\\
			\hline
			(0.9,\ 0.1,\ 0.1) & (0.1,\ 0.8,\ 0.1)& (0,\ 0.1,\ 0.8) & 0.831906&0.832529\\
			\hline
			(1,\ 0,\ 0) & (0,\ 1,\ 0)& (0,\ 0,\ 1) & 0.831914&0.831933\\
			\hline 
		\end{tabular}
		\caption{Density of visible  points in twisted P\'{o}lya's walk on  $\mathbb{Z}^3$}\label{tab:twisted walk in three}
	\end{table}
	We find that, on  one hand, the numerical values of the corresponding densities in the above two tables are very close to $1/\zeta(2)$ and $1/\zeta(3)$, respectively; on the other hand, as we have already mentioned above, the twisted walk is the standard walk for $\boldsymbol{\gamma}_r={\bf e}_{r},\ 1\leq r\leq k$,  the numerical calculations show that as we vary the values of $\boldsymbol{\gamma}_r$ away from ${\bf e}_r$, we didn't observe any dramatic change  compared to the standard P\'{o}lya's walk of which case we have proved theoretically.
	Hence we propose the following conjecture.
	\begin{conjecture}\label{thm:R_N^'=}
		For any integer $k\geq 2$ and any given initial point ${\bf q}_0\in\mathbb{Z}^k_{>0}$, the density of visible lattice points visited by the twisted  P\'{o}lya's walk in Definition \ref{def:twisted polya walk} is almost surely
		${1}/{\zeta(k)}$.
	\end{conjecture}
	
	Here we give a heuristic why we believe this conjecture. We may consider this twisted P\'{o}lya's walk in the following new coordinate system, 
	\[ y_j:=\boldsymbol{\gamma}_j\cdot (x_1, x_2, \ldots, x_k),\ 1\le j\le k. \]
	Let ${\bf q}'_{i}=(q'_{i,1}, \ldots, q'_{i,k})$ be the $i$-th step coordinate of the walker in the new coordinate system. Then the probability distribution of walking to each direction is the same as in the standard P\'{o}lya's walk in the new system, i.e.
	\[ \mathbb{P}(\text{walking to the}~j\text{-th direction})= \frac{q'_{i,j}}{q'_{i,1}+\cdots+q'_{i,k}},\ 1\le j\le k.\]
	However, after changing of variables, the possible waking directions in the new system are not the same as the positive axis directions anymore. This is where the difficulty is. Instead of considering visible steps in the first quadrant, in the new system we need to consider the distribution of visible steps in a restricted sector. In a sector shape, the density of all visible lattice points should be the same as in the whole plane. Hence, it is reasonable to believe this density result is also true for visible steps in such a P\'{o}lya's walk. Moreover, numerical experiments give strong support for our conjecture.

	\vspace{0.3cm}

	\noindent\textbf{Notations.} As usual, $\mathbb{C}$, $\mathbb{R}$, $\mathbb{Z}$ and $\mathbb{Z}_{>0}$ denote the sets of complex numbers, real numbers, integers and positive integers, respectively; $\Re(s)$ means the real part of a complex number $s$;  $\varepsilon$ denotes a sufficiently small positive number and $p$ represents a prime.

 We apply $\log (x)$ for the logarithmic function $\ln x$ and $\gcd(m_1,\cdots,m_l)$ for the greatest common divisor of integers $m_1,\cdots,m_l$ which are not all zero.
 
  We  use $\mathbb{P}$, $\mathbb{E}$ and $\mathbb{V}$ to mean taking probability, expectation and variances, respectively.
	
	The expression $f(n)=O(g(n))$ (or $f\ll g$) as $n\rightarrow\infty$ means there exist an integer $M\geq 1$ and a constant $C>0$ such that $|f(n)|\leq Cg(n)$ for $n\geq M$ . When the constant $C$ depends on some parameters ${\bf \rho}$, we write $f=O_{\bf \rho}(g)$(or $f\ll_\rho g$) as $n\rightarrow\infty$.
	

	Let $k\geq 2$ be an integer, the vector ${\bf e}_r\in\mathbb{Z}^k,\ 1\leq r\leq k$, denotes the  $k$-dimensional vector whose $r$th component is 1 and 0 elsewhere.
	
	For vector ${\bf v}=(v_1,\cdots,v_k)\in\mathbb{Z}^k$, we always write $s({\bf v}):=v_1+\cdots+v_k$.
	
	For integer $k\geq 2$ and nonnegative integers $n,\ u_1,\cdots,u_k$ with $u_1+\cdots+u_k=n$, we always write the multinomial coefficient
	$$
	\binom{n}{u_1,\cdots,u_k}:=\frac{n!}{u_1!\cdots u_k!}.
	$$
	As a convention, the above formula is of value $1$ for $n=0$.
	Besides, assume ${\boldsymbol{\alpha}}:=(\alpha_1,\cdots,\alpha_{k-1})$ satisfying $0<\alpha_1,\cdots,\alpha_{k-1}<1$ and $\alpha_1+\cdots+\alpha_{k-1}<1$ and ${\bf u}=(u_1,\cdots,u_k)$, we denote
	\begin{equation}\label{eq:def of P_{n,u}}
		P_{n,{\bf u},{\boldsymbol{\alpha}}}:=\binom{n}{u_1,\cdots,u_k}\alpha_1^{u_1}\cdots\alpha_{k-1}^{u_{k-1}}\Big(1-\sum_{r=1}^{k-1}\alpha_r\Big)^{u_k}.
	\end{equation}

	\section{Preliminaries}\label{sec:preliminaries}
	
	\subsection{Review of number-theoretic functions and related results} 
	
	For M\"{o}bius function $\mu(d)$, $d\in\mathbb{Z}_{>0}$, it is known that
	\begin{equation}\label{eq:the mobius inversion formula}
		\sum_{d\mid n}\mu(d)=\left\{
		\begin{aligned}
			&1, \ \ n=1,\\
			&0,\ \ n>1.
		\end{aligned}
		\right.
	\end{equation}
	Further, for integer $l\geq 2$, we have the identity
	\begin{equation}\label{eq:the mobius and zeta}
		\sum_{d=1}^\infty\frac{\mu(d)}{d^l}=\frac{1}{\zeta(l)}.
	\end{equation}

	For divisor function $\tau(d)=\sum_{t\mid d}1$,
	one may refer to page 296 of \cite{A} to get
	an upper bound $\tau(d)\ll_\varepsilon d^{\varepsilon}$ for any $\varepsilon>0$.

	The Euler gamma function $\Gamma(s)$, $s\in\mathbb{C}$, has the functional equation
	$\Gamma(s+1)=s\Gamma(s)$.
	Moreover, $\Gamma(s)$ satisfies the well-known Dirichlet integral formula (see \cite{D-2} and \cite{GR}), namely, for integer $l\geq 1$ and real numbers $b_1,\cdots,b_{l+1}>0$, we have
	
	\begin{equation}\label{eq:Dirichlet integral (1)}
		\frac{\Gamma(b_1)\cdots\Gamma(b_{l+1})}{\Gamma(b_1+\cdots+b_{l+1})}=\int_{\rm{T}}t_1^{b_1-1}\cdots t_{l}^{b_{l}-1}\Big(1-\sum_{m=1}^{l}t_m\Big)^{b_{l+1}-1}dt_1\cdots dt_{l},
	\end{equation}
	where the integration area $T$ is
	\begin{equation}\label{eq: the area T=}
		{\rm{T}}:=\bigg\{(t_1,\cdots,t_{l}):\ 0<t_m<1,\ \sum_{m=1}^{l}t_m<1\bigg\}.
	\end{equation}
	The right hand side of the above \eqref{eq:Dirichlet integral (1)} is the so-called  multivariate Beta function:
	\begin{equation}\label{eq: the def of beta}
		{{\rm{Beta}}({\bf b})}=\int_{\rm{T}}t_1^{b_1-1}\cdots t_{l}^{b_{l}-1}\Big(1-\sum_{m=1}^{l}t_m\Big)^{b_{l+1}-1}dt_1\cdots dt_{l},
	\end{equation}
	where ${\bf b}=(b_1,\cdots,b_{l+1})$.
	
	\subsection{The Dirichlet distribution} 
	The Dirichlet distribution is also known as the multivariate Beta distribution. We give the following definition. Further properties of the Dirichlet distribution can be found in \cite{BFG} and \cite{WTT}.
	\begin{definition}
		Let  $l\geq1$ be an integer and ${\bf a}=(a_1,\cdots,a_{l+1})$ with all $a_m>0$. A continuous random vector ${\bf T}=(T_1,\cdots,T_{l})$ has a Dirichlet distribution of dimension $l$, with parameter ${\bf a}$, if the probability density function ${\rm{Di}}_l({\bf t}|{{\bf a}})$ is 
		$$
		{\rm{Di}}_l({\bf t}|{{\bf a}})=\prod_{m=1}^lt_m^{a_m-1}\Big(1-\sum_{m=1}^{l}t_m\Big)^{a_{l+1}-1}\frac{1}{\rm{Beta}({\bf a})},
		$$
		where
		${\bf t}=(t_1,\cdots,t_{l})$ satisfying $0<t_m<1$ for any $1\leq m\leq l$ and $\sum_{m=1}^lt_m<1$. 
	\end{definition}
	
	Let $t_{l+1}:=1-\sum_{m=1}^{l}t_m$, and ${\rm T}$ be given by \eqref{eq: the area T=}, then  for  $m\in\{1,\cdots,l\}$,
	\begin{equation}\label{eq: bound for Dir for -1/2}
		\int_{\rm{T}}(t_mt_{l+1})^{-1/2}{\rm{Di}}_l({\bf t}|{{\bf a}})dt_1\cdots dt_{l}<\infty
	\end{equation}
	if and only if $a_m>1/2$ and $a_{l+1}>1/2$; and
	\begin{equation}\label{eq: bound for Dir for -1}
		\int_{\rm{T}}(t_mt_{l+1})^{-1}{\rm{Di}}_l({\bf t}|{{\bf a}})dt_1\cdots dt_{l}<\infty
	\end{equation} 
	if and only if $a_m>1$ and $a_{l+1}>1$.

	\subsection{Some useful lemmas from number theory}
	This subsection gives several results needed for this paper. We begin with the criteria of visibility of lattice points on $\mathbb{Z}^k$.
	
	\begin{lemma}[\cite{R-thesis}] \label{lem:the critia of visibility}
		For any integer $k\geq 2$ and ${\bf n}=(n_1,\cdots,n_k)\in\mathbb{Z}^k$, the lattice point ${\bf n}$ is visible  if and only if $\gcd(n_1,\cdots,n_k)=1$.
	\end{lemma}
	The next two lemmas are used several times in Section \ref{sec:proof of E(s) and V(s)}.
	\begin{lemma}[\cite{C(1)}] \label{lem: sum 1/n}
		We have
		$$
		\sum_{1\leq n\leq x}n^{-1}=\log x+\gamma+O(x^{-1})
		$$
		as $x\rightarrow\infty$, where $\gamma$ is Euler's constant, and for any real number $\theta>1$ and $x\geq1$,
		$$
		\sum_{1\leq n\leq x}n^{-1/2}=O(x^{1/2}) ,\ \ \ \sum_{n>x}n^{-\theta}=O_{\theta}(x^{1-\theta}).
		$$
	\end{lemma}

	\begin{lemma}[\cite{LM-1}, Lemma 2.9]\label{lem:the sum over i<j}
		For any real number $x\geq 2$, we have
		$$
		\sum_{1\leq n<m\leq x}n^{-1/2}=O(x^{3/2})
		\quad{\text{and}}\quad
		\sum_{1\leq n<m\leq x}(m-n)^{-1/2}=O(x^{3/2}).
		$$
	\end{lemma}

	\section{Density of visible lattice points in perturbed P\'{o}lya's walk on $\mathbb{Z}^{k}$}\label{sec:density in polya walk}
	This section gives  proof of Theorem \ref{thm:R(S)=}.
	Without loss generality, we only deal with the perturbed P\'{o}lya's walk with the constant $B=1$ in Definition \ref{Defn-Polya-walk}.
	We start with the following lemma,  which is essentially the second moment method from probability theory.
	\begin{lemma}[\cite{FF}, Proposition 8]\label{lem:lim R_n=u}
		Let $(X_i)_{i\geq1}$ be a sequence of uniformly bounded random variables, and let
		$$
		\overline{S}_N=\frac{1}{N}\sum_{1\leq i\leq N}X_i.
		$$
		If $\lim_{N\rightarrow\infty}\mathbb{E}(\overline{S}_N)=\rho$, and there exists a constant $\delta>0$ such that the variance $\mathbb{V}(\overline{S}_N)=O(N^{-\delta})$ for $N\geq1$, then we have
		$$	\lim_{N\rightarrow\infty}\overline{S}_N=\rho
		$$
		almost surely.
	\end{lemma}	
	By Lemma \ref{lem:lim R_n=u}, the main work is to compute the mean and variance of $\overline{R}_k({N})$. Before that, we consider the joint mass function of the sequence of jumps $({\bf Y}_i)_{i\geq 1}$ defined in Definition \ref{Defn-Polya-walk}. Let ${\bf e}_r\in\mathbb{Z}^k,\ 1\leq r\leq k$, be the  $k$-dimensional  vector whose $r$th component is 1 and 0 elsewhere. We have the following result.
	
	\begin{lemma}\label{lem:the joint mass function} 
		Suppose ${\boldsymbol{\alpha}}:=(\alpha_{1},\cdots,\alpha_{k-1})$ with $0<\alpha_{1},\cdots,\alpha_{k-1}<1$ and $\alpha_{1}+\cdots+\alpha_{k-1}<1$, then for any integer $n\geq 1$ and ${\bf y}_1,\cdots,{\bf y}_n\in\{{\bf e}_1,\cdots,{\bf e}_k\}$, we have
		$$
		\mathbb{P}\big({\bf Y}_1={\bf y}_1,\cdots,{\bf Y}_n={\bf y}_n\big)=\int_{\rm{A}}\alpha_1^{t_n^{(1)}}\cdots\alpha_{k-1}^{t_n^{(k-1)}}\Big(1-\sum_{r=1}^{k-1}\alpha_r\Big)^{t_n^{(k)}}{\rm{Di}}_{k-1}({\boldsymbol{\alpha}}|{\bf p}_0+{\boldsymbol{\beta}})d\alpha_1\cdots d\alpha_{k-1},
		$$
		where $t_n^{(r)}$ is the number of ${\bf e}_r$ among ${\bf y}_1,\cdots,{\bf y}_n$, and
		\begin{equation}\label{eq:the def of area A}
			{\rm{A}}:=\bigg\{{\boldsymbol{\alpha}}=(\alpha_1,\cdots,\alpha_{k-1}):\ 0<\alpha_r<1,\ \sum_{r=1}^{k-1}\alpha_r<1\bigg\}.
		\end{equation}
	\end{lemma}
	Lemma \ref{lem:the joint mass function} tells us that, the joint mass  function of these ${\bf Y}_i$ only depends on the number of ${\bf e}_1,\cdots,{\bf e}_k$ among ${\bf Y}_1,\cdots,{\bf Y}_n$ but not on the order of ${\bf Y}_i$.
	\begin{proof}
		We  write ${\bf y}_i:=(y_{i,1},\cdots,y_{i,k})$ with $y_{i,r}\in\{0,1\}$ and $y_{i,1}+\cdots+y_{i,k}=1$, then the conditional probability formula gives
		$$
		\begin{aligned}
			\mathbb{P}\big({\bf Y}_1={\bf y}_1,\cdots,{\bf Y}_n={\bf y}_n\big)&=\mathbb{P}\big({\bf Y}_1={\bf y}_1\big)\mathbb{P}\big({\bf Y}_2={\bf y}_2\ |\ {\bf Y}_1={\bf y}_1\big)\cdots\\
			& \mathbb{P}\big({\bf Y}_n={\bf y}_n\ |\ {\bf Y}_{n-1}={\bf y}_{n-1},\cdots, {\bf Y}_1={\bf y}_1\big).
		\end{aligned}
		$$
		Hence by the definition of ${\bf Y}_i$, we obtain
		\begin{align*}
			&\mathbb{P}\big({\bf Y}_1
			={\bf y}_1,\cdots,{\bf Y}_n={\bf y}_n\big)\\
			&\qquad=\prod_{r=1}^k \Big(\frac{p_{0,r}+\beta_r}{s({\bf p}_0)}\Big)^{y_{1,r}}\prod_{r=1}^k \Big(\frac{p_{0,r}+\beta_r+y_{1,r}}{s({\bf p}_0)+1}\Big)^{y_{2,r}}\cdots \prod_{r=1}^k \Big(\frac{p_{0,r}+\beta_r+\sum_{1\leq i\leq n-1}y_{i,r}}{s({\bf p}_0)+n-1}\Big)^{y_{n,r}}.	
		\end{align*}
		Then we infer that
		$$
		\mathbb{P}\big({\bf Y}_1={\bf y}_1,\cdots,{\bf Y}_n={\bf y}_n\big)=\frac{\prod_{j=0}^{t_n^{(1)}-1}(p_{0,1}+\beta_1+j)\cdots\prod_{j=0}^{t_n^{(k)}-1}(p_{0,k}+\beta_k+j)}{\prod_{j=0}^{n-1}\big(s({\bf p}_0)+j\big)},
		$$
		where $t_n^{(r)}$ ($1\leq r\leq k$) is the number of  1's in $y_{1,r},\cdots,y_{n,r}$ and is also the number of ${\bf e}_r$ among ${\bf y}_1,\cdots,{\bf y}_n$.

		From the equation $\Gamma(s+1)=s\Gamma(s)$ and the fact that $\beta_1+\cdots+\beta_k=0$, we represent the above joint mass function as
		$$
		\mathbb{P}\big({\bf Y}_1={\bf y}_1,\cdots,{\bf Y}_n={\bf y}_n\big)=\frac{\Gamma\big(p_{0,1}+\beta_1+t_n^{(1)}\big)\cdots\Gamma\big(p_{0,k}+\beta_k+t_n^{(k)}\big)\Gamma\big(s({\bf p}_0)\big)}{\Gamma\big(s({\bf p}_0)+n\big)\Gamma\big(p_{0,1}+\beta_1\big)\cdots\Gamma\big(p_{0,k}+\beta_k\big)}.
		$$
		Using formula \eqref{eq:Dirichlet integral (1)} and \eqref{eq: the def of beta}, we deduce 
		$$
		\mathbb{P}\big({\bf Y}_1={\bf y}_1,\cdots,{\bf Y}_n={\bf y}_n\big)=\frac{{\rm{Beta}}\big({\bf p}_0+{\boldsymbol{\beta}}+(t_n^{(1)},\cdots,t_n^{(k)})\big)}{{\rm{Beta}}\big({\bf p}_0+{\boldsymbol{\beta}}\big)}.
		$$
		We get immediately from the definition of Dirichlet distribution and \eqref{eq: the def of beta} that
		$$
		\mathbb{P}\big({\bf Y}_1={\bf y}_1,\cdots,{\bf Y}_n={\bf y}_n\big)=\int_{\rm{A}}\alpha_1^{t_n^{(1)}}\cdots\alpha_{k-1}^{t_n^{(k-1)}}\Big(1-\sum_{r=1}^{k-1}\alpha_r\Big)^{t_n^{(k)}}{\rm{Di}}_{k-1}({\boldsymbol{\alpha}}|{\bf p}_0+{\boldsymbol{\beta}})d\alpha_1\cdots d\alpha_{k-1},
		$$
		where ${\rm{A}}$ is defined in \eqref{eq:the def of area A}, which is our desired result.
	\end{proof}
	\begin{remark}
		By the definition of perturbed P\'{o}lya's walk,
		we see that the $n$th step ${\bf p}_n={\bf p}_0+\sum_{i=1}^n{\bf Y}_i$, this implies ${\bf p}_n$ is of the form ${\bf p}_n={\bf p}_0+(u_1,\cdots,u_k)$ for some  $0\leq u_1,\cdots,u_k\leq n$ with $u_1+\cdots+u_k=n$, where $u_r$ denotes the number of ${\bf e}_r$ among ${\bf Y}_1,\cdots,{\bf Y}_n$, hence
		\begin{equation}
			\nonumber\mathbb{P}\big({\bf p}_n={\bf p}_0+(u_1,\cdots,u_k)\big)=\mathbb{P}\Big(\sum_{i=1}^n{\bf Y}_i=(u_1,\cdots,u_k)\Big)
			=\binom{n}{u_1,\cdots,u_k}\mathbb{P}\big({\bf Y}_1,\cdots,{\bf Y}_n\big).
		\end{equation}
		Then  Lemma \ref{lem:the joint mass function} ensures that
		\begin{align}\label{eq:p_n=p_0+(u_1,u_k)}
			\mathbb{P}\big({\bf p}_n={\bf p}_0+\big(u_1,\cdots,u_k)\big)
			=\int_{\rm{A}}P_{n,{\bf u},{\boldsymbol{\alpha}}}{\rm{Di}}_{k-1}({\boldsymbol{\alpha}}|{\bf p}_0+{\boldsymbol{\beta}})d\alpha_1\cdots d\alpha_{k-1},
		\end{align}
		where ${\bf u}=(u_1,\cdots,u_k)$ and $P_{n,{\bf u},{\boldsymbol{\alpha}}}$ is defined by \eqref{eq:def of P_{n,u}}. Similarly, for $m>n\geq 1$ and some $0\leq v_1,\cdots,v_k\leq m-n$ with $v_1+\cdots+v_k=m-n$, we have
		$$
		\begin{aligned}
			\nonumber\mathbb{P}\big({\bf p}_n&={\bf p}_0+(u_1,\cdots,u_k),\ {\bf p}_m={\bf p}_n+(v_1,\cdots,v_k)\big)\\
			\nonumber&=\mathbb{P}\Big(\sum_{i=1}^n{\bf Y}_i=(u_1,\cdots,u_k),\ \sum_{i=n+1}^m{\bf Y}_i=(v_1,\cdots,v_k)\Big)\\
			&=\binom{n}{u_1,\cdots,u_k}\binom{m-n}{v_1,\cdots,v_k}\mathbb{P}\big({\bf Y}_1,\cdots,{\bf Y}_m\big),
		\end{aligned}
		$$
		which shows that, for  ${\bf v}=(v_1,\cdots,v_k)$,
		\begin{align}\label{eq:p_m=p_n+v}
			\nonumber\mathbb{P}\big({\bf p}_n&={\bf p}_0+(u_1,\cdots,u_k),\ {\bf p}_m={\bf p}_n+(v_1,\cdots,v_k)\big)\\
			&=\int_{\rm{A}}P_{n,{\bf u},{\boldsymbol{\alpha}}}P_{m-n,{\bf v},{\boldsymbol{\alpha}}}{\rm{Di}}_{k-1}({\boldsymbol{\alpha}}|{\bf p}_0+{\boldsymbol{\beta}})d\alpha_1\cdots d\alpha_{k-1}.
		\end{align}
	\end{remark}

	With the help of Lemma \ref{lem:the joint mass function}, we have the following results for the expectation and variance of $\overline{R}_k({N})$. The proofs are given in Section \ref{sec:proof of E(s) and V(s)}.
	
	\begin{proposition}\label{pro:E(R)} 
		Let  $k\geq 2$ be an integer. Suppose
		${\bf p}_0=(p_{0,1},\cdots,p_{0,k})$ is the starting point of the perturbed walk with all $p_{0,r}\geq 2$, and the perturbation $\boldsymbol{\beta}=(\beta_1,\cdots,\beta_k)$ satisfying all $|\beta_r|<1$. Then for any $\varepsilon>0$, we have, as $N\rightarrow\infty$,
		$$
		\mathbb{E}\big(\overline{R}_k({N})\big)=\frac{1}{\zeta(k)}+O_{k,\varepsilon,{\bf p}_0}(N^{-1/2+\varepsilon}).
		$$
	\end{proposition}

	\begin{proposition}\label{pro:V(R)=} 
			Let  $k\geq 2$ be an integer. Suppose
		${\bf p}_0=(p_{0,1},\cdots,p_{0,k})$ is the starting point of the perturbed walk with all $p_{0,r}\geq 2$, and the perturbation $\boldsymbol{\beta}=(\beta_1,\cdots,\beta_k)$ satisfying all $|\beta_r|<1$. Then for any $\varepsilon>0$, we have, as $N\rightarrow\infty$,
		$$
		\mathbb{V}\big(\overline{R}_k({N})\big)=O_{k,\varepsilon,{\bf p}_0}(N^{-1/2+\varepsilon}).
		$$
	\end{proposition}
	The rest of this subsection is devoted to prove Theorem \ref{thm:R(S)=} by the above two propositions.  
	\begin{proof}[Proof of Theorem {\rm \ref{thm:R(S)=}}]
		For the initial point ${\bf p}_0=(p_{0,1},\cdots,p_{0,k})$ satisfying $p_{0,r}\geq 2$ for all $1\leq r\leq k$, the result holds by Propositions \ref{pro:E(R)} and \ref{pro:V(R)=}, and Lemma \ref{lem:lim R_n=u}.
		
		If there exists some $t\in\{1,\cdots,k\}$ such that $p_{0,t}=1$, without loss of generality, we suppose $p_{0,k}=1$ and $p_{0,r}\geq 1$ for any $1\leq r\leq k-1$. Next we want to prove that the probability of ${\bf p}_n=(p_{n,1},\cdots,p_{n,k-1},1)$ tends to 0 as $n$ approaches $\infty$. Note that ${\bf p}_n$ can be written as ${\bf p}_n={\bf p}_0+\big(u_1,\cdots,u_k)$ for some $0\leq u_1,\cdots,u_{k}\leq n$ and $u_1+\cdots+u_{k}=n$, hence this is equivalent to prove $\mathbb{P}\big({\bf p}_n={\bf p}_0+\big(u_1,\cdots,u_{k-1},0)\big)=0$ as $n\rightarrow\infty$.  
		
		By \eqref{eq:p_n=p_0+(u_1,u_k)}, we have
		$$
		\mathbb{P}\big({\bf p}_n={\bf p}_0+\big(u_1,\cdots,u_{k-1},0)\big)
		=\frac{n!}{u_1!\cdots u_{k-1}!}\int_{\rm{A}}\alpha_1^{u_1}\cdots\alpha_{k-1}^{u_{k-1}}{\rm{Di}}_{k-1}({\boldsymbol{\alpha}}|{\bf p}_0+{\boldsymbol{\beta}})d\alpha_1\cdots d\alpha_{k-1}
		$$
		for some $0\leq u_1,\cdots,u_{k-1}\leq n$ and $u_1+\cdots+u_{k-1}=n$. The Dirichlet integral formula \eqref{eq:Dirichlet integral (1)} gives us that
		$$
		\begin{aligned}
			\mathbb{P}\big({\bf p}_n={\bf p}_0+\big(u_1,\cdots,u_{k-1},0)\big)
			=\frac{n!}{u_1!\cdots u_{k-1}!}\frac{\prod_{r=1}^{k-1}\Gamma(u _{r}+p_{0,r}+\beta_{r})\Gamma(\beta_k+1)}{\Gamma(n+1+p_{0,1}+\cdots+p_{0,k-1})}\frac{1}{\rm{Beta}({\bf p}_0+{\boldsymbol{\beta}})}.
		\end{aligned}
		$$
		Using the estimate
		\begin{equation}\label{eq:the bound of gamma (s+u)}
		\frac{\Gamma(s+u)}{\Gamma(s)}\ll(|s|+1)^{\kappa}e^{\frac{\pi}{2}|u|}
		\end{equation}
		with $|\arg s|\leq \pi-\varepsilon$, $\Re (s)=\sigma>0$ and $\Re (u)=\kappa>-\sigma$, we obtain
		$$
		\frac{\Gamma(u _{r}+p_{0,r}+\beta_{r})}{u_r!}=\frac{\Gamma(u _{r}+1+p_{0,r}+\beta_{r}-1)}{\Gamma(u_r+1)}\ll (u_r+2)^{p_{0,r}+\beta_{r}-1}
		$$
		for each $u_r\geq 0$, and for integer $n\geq 1$, we have
		$$
		\frac{n!}{\Gamma(n+1+p_{0,1}+\cdots+p_{0,k-1})}\ll(n+2+p_{0,1}+\cdots+p_{0,k-1})^{-(p_{0,1}+\cdots+p_{0,k-1})}.
		$$
		Combining all above gives, we deduce
		$$
		\begin{aligned}
			\mathbb{P}\big({\bf p}_n={\bf p}_0+\big(u_1,\cdots,u_{k-1},0)\big)\ll\frac{\prod_{r=1}^{k-1}(u_r+2)^{p_{0,r}+\beta_{r}}}{(n+2)^{p_{0,1}+\cdots+p_{0,k-1}}}\frac{1}{\prod_{r=1}^{k-1}(u_r+2)}
			\ll \frac{1}{(n+2)^{\beta_k}}\frac{1}{\prod_{r=1}^{k-1}(u_r+2)},
		\end{aligned}
		$$
		where the fact $\beta_1+\cdots+\beta_k=0$ is used for the last estimate.
		
		Since $u_1+\cdots+u_{k-1}=n$, there exists some $r_0\in\{1,\cdots,k-1\}$ satisfies $n/(k-1)\leq u_{r_0}\leq n$, then we have
		$$
		\mathbb{P}\big({\bf p}_n={\bf p}_0+\big(u_1,\cdots,u_{k-1},0)\big)\ll \frac{1}{(n+2)^{\beta_k}}\frac{1}{(u_{r_0}+2)}\ll_k \frac{1}{(n+2)^{1+\beta_k}}
		$$
		It follows from our assumption  $|\beta_r|<1$ for all $r$ that
		$$
		\mathbb{P}\big({\bf p}_n={\bf p}_0+\big(u_1,\cdots,u_{k-1},0)\big)\rightarrow0,\ \ \ n\rightarrow\infty,
		$$
		which is what we want to prove.
		
		From this we conclude that the walker will reach a new  point ${\bf p}_0^{\prime}$ satisfying $p_{0,k}^{\prime}\geq 2$ with probability one. At this moment, if all $p_{0,r}^{\prime}\geq 2$, then Theorem \ref{thm:R(S)=} holds; otherwise, we repeat the above process,  then the walker eventually enters the region $\{(m_1,\cdots,m_k):\ m_r\geq 2,\ \forall 1\leq r\leq k\}$ with probability 1. This gives Theorem \ref{thm:R(S)=}.
	\end{proof}

	\section{Compute the expectation and variance of \(\overline{R}_{k}(N)\)}\label{sec:proof of E(s) and V(s)}
	Our goal in this section is to prove  Propositions \ref{pro:E(R)} and \ref{pro:V(R)=}. To get started, we introduce the following useful lemma which is a crucial ingredient in this paper.
	\begin{lemma}\label{lem:sum over u}
		Let $n\geq 1,\ k\geq 2$ be integers and $c_1,\cdots,c_{k-1}\in\mathbb{Z}$. Assume ${\boldsymbol{\alpha}}:=(\alpha_{1},\cdots,\alpha_{k-1})$ satisfying $0<\alpha_{1},\cdots,\alpha_{k-1}<1$ and $\alpha_{1}+\cdots+\alpha_{k-1}<1$. Then for any integer $d\geq 1$, we have
		$$
		\sum_{\substack{0\leq u_1,\cdots,u_k\leq n\\u_1+\cdots+u_k= n\\ u_r\equiv c_r(\bmod d),\\ \forall 1\leq r\leq k-1}}\binom{n}{u_1,\cdots,u_k}\alpha_1^{u_1}\cdots\alpha_{k-1}^{u_{k-1}}\alpha_{k}^{u_k}=\frac{1}{d^{k-1}}+\sum_{1\leq t\leq k-1}\frac{1}{\sqrt{\alpha_t\alpha_{k}}}O_{k}\Big(\frac{\log n}{\sqrt{n}}\Big)
		$$
		as $n\rightarrow\infty$, where $\alpha_{k}:=1-\sum_{r=1}^{k-1}\alpha_r$.
	\end{lemma}
	
	
	The proof of Lemma \ref{lem:sum over u} is a little bit complicated, and we give the proof in Section \ref{sec:appendix}. First we use it to prove Propositions \ref{pro:E(R)} and \ref{pro:V(R)=}.

	\subsection{The asymptotic formula for the expectation of $\overline{R}_k({N})$}
	We write Proposition \ref{pro:E(R)} again for the convenience
	of the reader:
	\vspace{0.3cm}	
	
	\noindent{\bf Proposition \ref{pro:E(R)}.}	\textit{Let  $k\geq 2$ be an integer. Suppose
		${\bf p}_0=(p_{0,1},\cdots,p_{0,k})$ is the starting point of the perturbed walk with all $p_{0,r}\geq 2$, and the perturbation $\boldsymbol{\beta}=(\beta_1,\cdots,\beta_k)$ satisfying all $|\beta_r|<1$. Then for any $\varepsilon>0$, we have, as $N\rightarrow\infty$,
		$$
		\mathbb{E}\big(\overline{R}_k({N})\big)=\frac{1}{\zeta(k)}+O_{k,\varepsilon,{\bf p}_0}(N^{-1/2+\varepsilon}).
		$$}

	\begin{proof}
		We recall from \eqref{def:V=} and \eqref{def:R=} that
		\begin{equation}\label{eq:E(S)=E(X)_n}
			\mathbb{E}\big(\overline{R}_k({N})\big)=\frac{1}{N}\sum_{1\leq n\leq N}\mathbb{E}({V}_n)=\frac{1}{N}\sum_{1\leq n\leq N}\mathbb{P}({\bf p}_n\ \rm{is}\ \rm{visible}).
		\end{equation}
		Then by \eqref{eq:p_n=p_0+(u_1,u_k)}, we obtain
		$$
		\begin{aligned}
			\mathbb{E}({V}_n)=\sum_{\substack{ u_{1},\cdots,u_{k}}}\mathbb{P}\big({\bf p}_n={\bf p}_0+(u_1,\cdots,u_k)\big)=\int_{\rm{A}}\sum_{\substack{ u_{1},\cdots,u_{k}}}P_{n,{\bf u},{\boldsymbol{\alpha}}}{\rm{Di}}_{k-1}({\boldsymbol{\alpha}}|{\bf p}_0+{\boldsymbol{\beta}})d\alpha_1\cdots d\alpha_{k-1}
		\end{aligned}
		$$
		with ${\bf u}=(u_1,\cdots,u_k)$, here, for notational simplicity, we write
		$$
		\sum_{\substack{ u_{1},\cdots,u_{k}}}:=\sum_{\substack{0\leq u_{1},\cdots,u_{k}\leq n\\u_{1}+\cdots+u_{k}=n\\  (u_1+p_{0,1},\cdots,u_k+p_{0,k})\ {\rm{is}}\ {\rm{visible}}}}.
		$$
		It is readily to seen from Lemma \ref{lem:the critia of visibility} that
		$$
		\sum_{\substack{ u_{1},\cdots,u_{k}}}P_{n,{\bf u},{\boldsymbol{\alpha}}}=\sum_{\substack{0\leq u_{1},\cdots,u_{k}\leq n\\u_{1}+\cdots+u_{k}=n\\ \gcd(u_1+p_{0,1},\cdots,u_k+p_{0,k})=1}}P_{n,{\bf u},{\boldsymbol{\alpha}}}=\sum_{\substack{0\leq u_{1},\cdots,u_{k}\leq n\\u_{1}+\cdots+u_{k}=n\\ \gcd(u_1+p_{0,1},\cdots,u_{k-1}+p_{0,k-1},n+s({\bf p}_0))=1}}P_{n,{\bf u},{\boldsymbol{\alpha}}},
		$$
		where $s({\bf v}):=v_1+\cdots+v_k$ for vector ${\bf v}=(v_1,\cdots,v_k)\in\mathbb{Z}^k$.
		Hence by the well-known formula \eqref{eq:the mobius inversion formula}, we may write
		$$
		\begin{aligned}
			\sum_{\substack{ u_{1},\cdots,u_{k}}}P_{n,{\bf u},{\boldsymbol{\alpha}}}=\sum_{d\mid n+s({\bf p}_0)}\mu(d)\sum_{\substack{0\leq u_{1},\cdots,u_{k}\leq n\\u_{1}+\cdots+u_{k}=n\\d\mid u_r+p_{0,r},\forall 1\leq r\leq k-1}}P_{n,{\bf u},{\boldsymbol{\alpha}}}.
		\end{aligned}
		$$
		We apply Lemma \ref{lem:sum over u} for the inner sum and use the fact $|\mu(d)|\leq 1$, there holds
		\begin{equation}\label{eq:E(Y_n=)}
			\sum_{\substack{ u_{1},\cdots,u_{k}}}P_{n,{\bf u},{\boldsymbol{\alpha}}}=\sum_{{ d\mid n+s({\bf p}_0)}}\frac{\mu(d)}{d^{k-1}}+\sum_{1\leq t\leq k-1}\frac{1}{\sqrt{\alpha_t\alpha_k}}O_k\big(n^{-1/2+\varepsilon}\tau\big(n+s({\bf p}_0)\big)\big),
		\end{equation}
		here we used the formula that $\log x=O_{\varepsilon}(x^\varepsilon)$ for any $\varepsilon>0$, $x\rightarrow\infty$. Recall that for any $\varepsilon>0$, the bound of divisor function is
		\begin{equation*}
			\tau\big(n+s({\bf p}_0)\big)\ll_\varepsilon\big(n+s({\bf p}_0)\big)^\varepsilon\ll_{\varepsilon,{\bf p}_0} n^\varepsilon,
		\end{equation*}
		which implies
		\begin{equation*}
			\sum_{\substack{ u_{1},\cdots,u_{k}}}P_{n,{\bf u},{\boldsymbol{\alpha}}}=\sum_{d\mid n+s({\bf p}_0)}\frac{\mu(d)}{d^{k-1}}+\sum_{1\leq t\leq k-1}\frac{1}{\sqrt{\alpha_t\alpha_k}}O_{k,\varepsilon,{\bf p}_0}(n^{-1/2+\varepsilon}).
		\end{equation*}
		Gathering the above results we obtain
		$$
		\mathbb{E}({V}_n)=\sum_{\substack{ d\mid n+s({\bf p}_0)}}\frac{\mu(d)}{d^{k-1}}+O_{k,\varepsilon,{\bf p}_0}(n^{-1/2+\varepsilon})\int_{\rm{A}}\sum_{1\leq t\leq k-1}\frac{1}{\sqrt{\alpha_t\alpha_k}}{\rm{Di}}_{k-1}({\boldsymbol{\alpha}}|{\bf p}_0+{\boldsymbol{\beta}})d\alpha_1\cdots d\alpha_{k-1}.
		$$
		Since $p_{0,r}\geq 2$ implies that $p_{0,r}+\beta_r>1$ for any $1\leq r\leq k$,  by \eqref{eq: bound for Dir for -1/2}, the above integral is convergent, which gives
		$$
		\mathbb{E}({V}_n)=\sum_{\substack{ d\mid n+s({\bf p}_0)}}\frac{\mu(d)}{d^{k-1}}+O_{k,\varepsilon,{\bf p}_0}(n^{-1/2+\varepsilon})
		$$
		for all $p_{0,r}\geq 2$.
		This combines \eqref{eq:E(S)=E(X)_n} and Lemma \ref{lem: sum 1/n} yields
		\begin{equation}\label{eq:E(s)=(2)}
			\mathbb{E}\big(\overline{R}_k({N})\big)=\frac{1}{N}\sum_{1\leq n\leq N}\sum_{\substack{d\mid n+s({\bf p}_0)}}\frac{\mu(d)}{d^{k-1}}+O_{k,\varepsilon,{\bf p}_0}(N^{-1/2+\varepsilon}).
		\end{equation}
		
		For the first term on the right hand of the above equation, we denote
		$$
		G_{N,k}:=	G_{N,k}({\bf p}_0)=\sum_{1\leq n\leq N}\sum_{\substack{d\mid n+s({\bf p}_0)}}\frac{\mu(d)}{d^{k-1}}=\sum_{\substack{1\leq d\leq N+s({\bf p}_0)}}\frac{\mu(d)}{d^{k-1}}\sum_{\substack{1\leq n\leq N\\n\equiv -s({\bf p}_0)(\bmod d)}}1.
		$$ 
		Observe that the number of $n\in\{1,\cdots,N\}$ satisfying $n\equiv -s({\bf p}_0)(\bmod d)$ has the asymptotic formula ${N}/{d}+O(1)$, then we arrive at
		$$
		G_{N,k}=N\sum_{\substack{1\leq d\leq N+s({\bf p}_0)}}\frac{\mu(d)}{d^{k}}+O\Big(\sum_{1\leq d\leq N+s({\bf p}_0)}\frac{1}{d}\Big).
		$$
		Appealing to Lemma \ref{lem: sum 1/n} for the $O$-term, we derive
		$$
		\begin{aligned}
			G_{N,k}&=N\sum_{\substack{1\leq d\leq N+s({\bf p}_0)}}\frac{\mu(d)}{d^{k}}+O_{\varepsilon,{\bf p}_0}(N^\varepsilon)\\
			&=N\sum_{d=1}^\infty\frac{\mu(d)}{d^k}+O\Big(N\sum_{d>N+s({\bf p}_0)}\frac{1}{d^k}\Big)+O_{\varepsilon,{\bf p}_0}(N^\varepsilon).
		\end{aligned}
		$$
		Then we get from Lemma \ref{lem: sum 1/n} and \eqref{eq:the mobius and zeta} that
		\begin{equation}\label{eq:G=}
			G_{N,k}=\frac{N}{\zeta(k)}+O_{k,\varepsilon,{\bf p}_0}(N^\varepsilon).
		\end{equation}
		This completes our proof by collecting \eqref{eq:E(s)=(2)} and \eqref{eq:G=}.
	\end{proof}

	\subsection{Estimating the variance of $\overline{R}_{k}(N)$} We recall Proposition \ref{pro:V(R)=} as follows:
	\vspace{0.3cm}
	
	\noindent{\bf Proposition \ref{pro:V(R)=}.} \textit{Let  $k\geq 2$ be an integer. Suppose
		${\bf p}_0=(p_{0,1},\cdots,p_{0,k})$ is the starting point of the perturbed walk with all $p_{0,r}\geq 2$, and the perturbation $\boldsymbol{\beta}=(\beta_1,\cdots,\beta_k)$ satisfying all $|\beta_r|<1$. Then for any $\varepsilon>0$, we have, as $N\rightarrow\infty$,
		$$
		\mathbb{V}\big(\overline{R}_k({N})\big)=O_{k,\varepsilon,{\bf p}_0}(N^{-1/2+\varepsilon}).
		$$}
	\begin{proof}
		We write the variance
		\begin{align}\label{eq:V(s(n))=E(s(n))^2-(E(s(n)))^2}
			\mathbb{V}\big(\overline{R}_k({N})\big)=\mathbb{E}\big(\overline{R}_k({N})^2\big)-\mathbb{E}\big(\overline{R}_k({N})\big)^2.
		\end{align}
		By virtue of  Proposition \ref{pro:E(R)}, the contribution of the last term to the variance of $\overline{R}_k({N})$ is
		\begin{equation}\label{eq:(E(s(n)))^2=}
			\mathbb{E}\big(\overline{R}_k({N})\big)^2=\Big(\frac{1}{\zeta(k)}\Big)^2+O_{k,\varepsilon,{\bf p}_0}(N^{-1/2+\varepsilon}).
		\end{equation}
		To compute $\mathbb{E}\big(\overline{R}_k({N})^2\big)$, we expand the square and obtain
		\begin{align}\label{eq:E((s(n))^2)=sum_E(X_iX_j)+sum_E(X_i^2)}
			\mathbb{E}\big(\overline{R}_k({N})^2\big)=\frac{2}{N^2}\sum_{1\leq n<m\leq N}\mathbb{E}({V}_n{V}_m)+\frac{1}{N^2}\sum_{1\leq n\leq N}\mathbb{E}({{V}_n}^2).
		\end{align}
		For the last sum of \eqref{eq:E((s(n))^2)=sum_E(X_iX_j)+sum_E(X_i^2)}, we have trivially
		\begin{align}\label{eq:sum EX_i^2=}
			\sum_{1\leq n\leq N}\mathbb{E}({{V}_n}^2)=\sum_{1\leq n\leq N}\mathbb{E}({V}_n)=\sum_{1\leq n\leq N}\mathbb{P}({\bf p}_n\ {\rm{is\ visible}})=O(N).
		\end{align}
		For $1\leq n<m\leq N$, note that
		$$
		\mathbb{E}({V}_n{V}_m)=\mathbb{P}({\bf p}_n\ {\rm{and}}\ {\bf p}_m\ \rm{are\ both}\ \rm{visible}).
		$$
		Recall from \eqref{eq:p_m=p_n+v} that
		\begin{align}\label{eq:E(V_nV_m=int)}
			\nonumber\mathbb{E}({V}_n{V}_m)&=\sum_{\substack{ u_{1},\cdots,u_{k}}}\sum_{\substack{v_{1},\cdots,v_{k}}}\mathbb{P}\big({\bf p}_n={\bf p}_0+(u_1,\cdots,u_k),\ {\bf p}_m={\bf p}_n+(v_1,\cdots,v_k)\big)\\
			&=\int_{\rm{A}}\sum_{\substack{ u_{1},\cdots,u_{k}}}\sum_{\substack{v_{1},\cdots,v_{k}}}P_{n,{\bf u},{\boldsymbol{\alpha}}}P_{m-n,{\bf v},{\boldsymbol{\alpha}}}{\rm{Di}}_{k-1}({\boldsymbol{\alpha}}|{\bf p}_0+{\boldsymbol{\beta}})d\alpha_1\cdots d\alpha_{k-1},
		\end{align}
		where 
		$$
		\sum_{\substack{ u_{1},\cdots,u_{k}}}\sum_{\substack{v_{1},\cdots,v_{k}}}:=\sum_{\substack{0\leq u_{1},\cdots,u_{k}\leq n\\u_{1}+\cdots+u_{k}=n\\  (u_1+p_{0,1},\cdots,u_k+p_{0,k})\ {\rm{is}}\ {\rm{visible}}}}\sum_{\substack{0\leq v_{1},\cdots,v_{k}\leq m-n\\v_{1}+\cdots+v_{k}=m-n\\ (u_1+p_{0,1}+v_1,\cdots,u_k+p_{0,k}+v_k)\ {\rm{is}}\ {\rm{visible}}}}.
		$$
		By a similar argument as that yields \eqref{eq:E(Y_n=)}, we obtain
		$$
		\sum_{\substack{v_{1},\cdots,v_{k}}}P_{m-n,{\bf v},{\boldsymbol{\alpha}}}=\sum_{\substack{d\mid m-n+s({\bf p}_0+{\bf u})}}\frac{\mu(d)}{d^{k-1}}+\sum_{1\leq t\leq k-1}\frac{1}{\sqrt{\alpha_t\alpha_k}}O_{k}\big((m-n)^{-1/2+\varepsilon}\tau(m-n+s({\bf p}_0+{\bf u}))\big).
		$$
		It then follows that
		$$
		\begin{aligned}
			\sum_{\substack{ u_{1},\cdots,u_{k}}}&\sum_{\substack{v_{1},\cdots,v_{k}}}P_{n,{\bf u},{\boldsymbol{\alpha}}}P_{m-n,{\bf v},{\boldsymbol{\alpha}}}\\
			&=\Big(\sum_{\substack{ d\mid n+s({\bf p}_0)}}\frac{\mu(d)}{d^{k-1}}+\sum_{1\leq t\leq k-1}\frac{1}{\sqrt{\alpha_t\alpha_k}}O_{k,\varepsilon,{\bf p}_0}\big(n^{-1/2+\varepsilon}\big)\Big)\cdot\\
			&\Big(\sum_{\substack{d\mid m+s({\bf p}_0)}}\frac{\mu(d)}{d^{k-1}}+\sum_{1\leq t\leq k-1}\frac{1}{\sqrt{\alpha_t\alpha_k}}O_{k,\varepsilon,{\bf p}_0}\big((m-n)^{-1/2+\varepsilon}m^\varepsilon\big)\Big)
		\end{aligned}
		$$
		for any $\varepsilon>0$. We  use the bound
		\begin{equation}\label{eq: bound of mu(d)/d^k}
			\sum_{\substack{d\mid n+s({\bf p}_0)}}\frac{\mu(d)}{d^{k-1}}\ll\sum_{\substack{d\mid n+s({\bf p}_0)}}\frac{1}{d^{k-1}}\ll\sum_{\substack{d\mid n+s({\bf p}_0)}}1\ll_{\varepsilon,{\bf p}_0} n^\varepsilon
		\end{equation}
		to expand the above formula, and for the sum related to $\alpha_t\alpha_k$, we apply the estimate $1/\sqrt{\alpha_t\alpha_k}\leq 1/\alpha_t\alpha_k$ and
		$$
		\Big(\sum_{1\leq t\leq k-1}\frac{1}{\sqrt{\alpha_t\alpha_k}}\Big)^2\ll_k \sum_{1\leq t\leq k-1}\frac{1}{\alpha_t\alpha_k},
		$$
		then there holds
		$$
		\begin{aligned}
			\sum_{\substack{ u_{1},\cdots,u_{k}}}&\sum_{\substack{v_{1},\cdots,v_{k}}}P_{n,{\bf u},{\boldsymbol{\alpha}}}P_{m-n,{\bf v},{\boldsymbol{\alpha}}}\\
			&=\sum_{\substack{d\mid n+s({\bf p}_0)}}\frac{\mu(d)}{d^{k-1}}\sum_{\substack{d\mid m+s({\bf p}_0)}}\frac{\mu(d)}{d^{k-1}}
			+\sum_{1\leq t\leq k-1}\frac{1}{{\alpha_t\alpha_k}}O_{k,\varepsilon,{\bf p}_0}\big(n^\varepsilon m^\varepsilon(m-n)^{-1/2+\varepsilon}+n^{-1/2+\varepsilon}m^\varepsilon\big).
		\end{aligned}
		$$
		
		Next we insert this identity into \eqref{eq:E(V_nV_m=int)}, and then use \eqref{eq: bound for Dir for -1} for the convergence of the integral, we deduce
		$$
		\begin{aligned}
			\mathbb{E}({V}_n{V}_m)=\sum_{\substack{d\mid n+s({\bf p}_0)}}\frac{\mu(d)}{d^{k-1}}\sum_{\substack{d\mid m+s({\bf p}_0)}}\frac{\mu(d)}{d^{k-1}}+O_{k,\varepsilon,{\bf p}_0}\big(n^\varepsilon m^\varepsilon(m-n)^{-1/2+\varepsilon}+n^{-1/2+\varepsilon}m^\varepsilon\big)
		\end{aligned}
		$$
		for $p_{0,r}\geq 2,\ \forall 1\leq r\leq k$.
		To continue the computation, we  sum over $1\leq n<m\leq N$, getting
		\begin{align*}
			\sum_{1\leq n<m\leq N}\mathbb{E}({V}_n{V}_m)=\sum_{1\leq n<m\leq N}\sum_{\substack{d\mid n+s({\bf p}_0)}}\frac{\mu(d)}{d^{k-1}}\sum_{\substack{d\mid m+s({\bf p}_0)}}\frac{\mu(d)}{d^{k-1}}+O_{k,\varepsilon,{\bf p}_0}(N^{3/2+\varepsilon}),
		\end{align*}
		thanks to Lemma \ref{lem:the sum over i<j}. To simplify the first term on the right side of the equation, by \eqref{eq: bound of mu(d)/d^k}, we may add diagonal terms up to an error term  $O_{\varepsilon,{\bf p}_0}(N^{1+\varepsilon})$, then
		$$
		\sum_{1\leq n<m\leq N}\mathbb{E}({V}_n{V}_m)=\frac{1}{2}\Big(\sum_{1\leq n\leq N}\sum_{\substack{d\mid n+s({\bf p}_0)}}\frac{\mu(d)}{d^{k-1}}\Big)^2+O_{k,\varepsilon,{\bf p}_0}(N^{3/2+\varepsilon}).
		$$
		We get immediately from \eqref{eq:G=} that
		\begin{equation}\label{eq:sum_EX_iEX_j=}
			\sum_{1\leq n<m\leq N}\mathbb{E}({V}_n{V}_m)=\frac{N^2}{2}\Big(\frac{1}{\zeta(k)}\Big)^2+O_{k,\varepsilon,{\bf p}_0}(N^{3/2+\varepsilon}).
		\end{equation}
		Gathering together all the pieces \eqref{eq:E((s(n))^2)=sum_E(X_iX_j)+sum_E(X_i^2)}, \eqref{eq:sum EX_i^2=} and \eqref{eq:sum_EX_iEX_j=}, we get
		$$
		\mathbb{E}\big(\overline{R}_k({N})^2\big)=\Big(\frac{1}{\zeta(k)}\Big)^2+O_{k,\varepsilon,{\bf p}_0}(N^{-1/2+\varepsilon}),
		$$
		which together with \eqref{eq:(E(s(n)))^2=} and \eqref{eq:V(s(n))=E(s(n))^2-(E(s(n)))^2} yields our required result.
	\end{proof}

	\section{The proof of key lemma}\label{sec:appendix}
	In this section, we provide a proof of Lemma \ref{lem:sum over u} by tools from analytic number theory. In order to prove Lemma \ref{lem:sum over u}, we need the following Lemmas \ref{lem:bound for binomial paobability}-\ref{lem:estimate of beta_{t_i}}. 

	\begin{lemma}\label{lem:bound for binomial paobability}
		For any integer $m\geq 1$ and $0<\alpha<1$, we have
		$$
		\alpha^m\ll\frac{1}{\sqrt{\alpha(1-\alpha)}}\frac{1}{\sqrt{m}}.
		$$
	\end{lemma}
	\begin{proof}
		The result is essentially from the bound for binomial probability in
		\cite{FF}
		$$
		\max_{0\leq t\leq m}\binom{m}{t}\alpha^t(1-\alpha)^{m-t}\leq C\frac{1}{\sqrt{m\alpha(1-\alpha)}},
		$$
		where $C>0$ is a certain constant.
	\end{proof}
	
	We use the following estimate in terms of cosine function.
	\begin{lemma}\label{lem:cos x<<}
		Let $d>2$ be an integer, then we have
		$$
		\sum_{1\leq h<d/2}\cos^l\big(\frac{2\pi h}{d}\big)\ll\sum_{1\leq h<d/2}\cos^l\big(\frac{\pi h}{d}\big)
		$$
		for integer $l\geq 0$, and for $l\geq1$, there holds
		$$
		\sum_{1\leq h<\frac{d}{2}}\cos^l\big(\frac{\pi h}{d}\big)\ll\frac{d}{\sqrt{l}}.
		$$
	\end{lemma}
	\begin{proof}
		For the first estimate, we only consider $l\geq 1$. Using change of variable $h^{\prime}=2h$, we obtain
		$$
		\sum_{1\leq h<d/2}\cos^l\big(\frac{2\pi h}{d}\big)=\sum_{\substack{1\leq h^{\prime}<d\\h^{\prime}\ {\rm{even}}}}\cos^l\big(\frac{\pi h^{\prime}}{d}\big).
		$$
		Spliting the  sum into two parts gives us
		$$
		\sum_{1\leq h<d/2}\cos^l\big(\frac{2\pi h}{d}\big)=\sum_{\substack{1\leq h^{\prime}<d/2\\h^{\prime}\ {\rm{even}}}}\cos^l\big(\frac{\pi h^{\prime}}{d}\big)+\sum_{\substack{d/2< h^{\prime}<d\\h^{\prime}\ {\rm{even}}}}\cos^l\big(\frac{\pi h^{\prime}}{d}\big).
		$$
		For the last summation, we write $t=d-h^{\prime}$ and obtain
		$$
		\sum_{1\leq h<d/2}\cos^l\big(\frac{2\pi h}{d}\big)=\sum_{\substack{1\leq h^{\prime}<d/2\\h^{\prime}\ {\rm{even}}}}\cos^l\big(\frac{\pi h^{\prime}}{d}\big)+(-1)^l\sum_{\substack{1\leq t<d/2\\t\ {\rm{even}}}}\cos^l\big(\frac{\pi t}{d}\big)
		$$
		for even $d$, and
		$$
		\sum_{1\leq h<d/2}\cos^l\big(\frac{2\pi h}{d}\big)=\sum_{\substack{1\leq h^{\prime}<d/2\\h^{\prime}\ {\rm{even}}}}\cos^l\big(\frac{\pi h^{\prime}}{d}\big)+(-1)^l\sum_{\substack{1\leq t<d/2\\t\ {\rm{odd}}}}\cos^l\big(\frac{\pi t}{d}\big)
		$$
		for odd $d$, which is our desired result.
		
		For the second assertion, since $\cos^l(\pi t/d)$ is decreasing for $t\in[0,d/2]$,  we have
		$$
		\sum_{1\leq h<\frac{d}{2}}\cos^l\big(\frac{\pi h}{d}\big)\ll\int_{0}^{\frac{d}{2}}\cos^l\big(\frac{\pi t}{d}\big)dt.
		$$
		Using change of variable $x=\pi t/d$ and the formula
		$$
		\int_{0}^{\frac{\pi}{2}}\cos^l(x)dx=\frac{\sqrt{\pi}\Gamma(\frac{l+1}{2})}{2\Gamma(\frac{l+2}{2})},
		$$
		where $\Gamma(s)$ is the Euler gamma function, we obtain
		\begin{align}\label{eq: sum_hcos^l<<}
			\sum_{1\leq h<\frac{d}{2}}\cos^l\big(\frac{\pi h}{d}\big)\ll d\int_{0}^{\frac{\pi}{2}}\cos^l(x)dx\ll\frac{d\Gamma(\frac{l+1}{2})}{\Gamma(\frac{l+2}{2})}.
		\end{align}
		Applying the estimate \eqref{eq:the bound of gamma (s+u)} with $s=(l+2)/2$ and $u=-1/2$, we arrive at
		$$
		\frac{\Gamma(\frac{l+1}{2})}{\Gamma(\frac{l+2}{2})}=\frac{\Gamma(\frac{l+2}{2}-\frac{1}{2})}{\Gamma(\frac{l+2}{2})}\ll\frac{1}{\sqrt{l}}.
		$$
		This together with \eqref{eq: sum_hcos^l<<} gives our desired result.
	\end{proof}
	
	\begin{lemma}\label{lem:the sum in two dimensions}
		Suppose $n\in\mathbb{Z}_{>0}$ and $0<\delta\leq \frac{1}{2}$. Then for any integer $ d>1$, we have
		$$
		\frac{1}{d}\sum_{1\leq h\leq d-1}\big(1-\delta+\delta\cos(\frac{2\pi h}{d})\big)^{\frac{n}{2}}=\frac{1}{\sqrt{\delta(1-\delta)}}O\Big(\frac{\log n}{\sqrt{n}}\Big)
		$$
		as $n\rightarrow\infty$.
	\end{lemma}
	
	\begin{proof}
		To simplify our notations, we denote
		$$
		H_{n,\delta}=H_{n,\delta}(d):=\frac{1}{d}\sum_{1\leq h\leq d-1}\big(1-\delta+\delta\cos(\frac{2\pi h}{d})\big)^{\frac{n}{2}}.
		$$
		For the case $d=2$, by virtue of Lemma \ref{lem:bound for binomial paobability}, we get
		$$
		H_{n,\delta}\ll (1-2\delta)^{\frac{n}{2}}\ll(1-\delta)^{\frac{n}{2}}\ll\frac{1}{\sqrt{\delta(1-\delta)}}\frac{1}{\sqrt{n}}.
		$$
		So we only need to handle the remaining case $d>2$.
		
		If $n=2m$ is even, we resort to the binomial theorem and obtain
		$$
			H_{n,\delta}=\frac{1}{d}\sum_{0\leq l\leq m}\binom{m}{l}(1-\delta)^{m-l}\delta^l \sum_{1\leq h\leq d-1}\cos^l\big(\frac{2\pi h}{d}\big).
		$$
		By an elementary argument, we derive
		$$
		\sum_{1\leq h\leq d-1}\cos^l\big(\frac{2\pi h}{d}\big)=2\sum_{1\leq h<d/2}\cos^l\big(\frac{2\pi h}{d}\big)+\epsilon_{d}(-1)^l,
		$$	
		where
		$$
		\epsilon_{d}=\begin{cases}
			1, \quad{\rm if}\ d\ {\rm\ is\ even},\\
			0, \quad{\rm if}\  d\ {\rm\ is\ odd}.
		\end{cases}
		$$
		Applying Lemma \ref{lem:cos x<<} and the binomial theorem again, we have
		$$
			H_{n,\delta}\ll\frac{1}{d}\sum_{0\leq l\leq m}\binom{m}{l}(1-\delta)^{m-l}\delta^l I(l,d)+(1-2\delta)^m,
		$$
		where	
		$$
		I(l,d):=\sum_{1\leq h<d/2}\cos^l\big(\frac{\pi h}{d}\big).
		$$
		Then Lemma \ref{lem:bound for binomial paobability} tells us that the bound of the last term is
		$$
		(1-2\delta)^m\ll (1-\delta)^m\ll \frac{1}{\sqrt{\delta(1-\delta)}}\frac{1}{\sqrt{m}}.
		$$
	We divide the sum over $l$ into two parts according to $l\leq m/\log^2m$ or not, then
		\begin{align}\label{eq:H_n<<H_n,1+H_n,2}
				H_{n,\delta}\ll 	H_{n,\delta}^{\prime}+	H_{n,\delta}^{\prime\prime}+\frac{1}{\sqrt{\delta(1-\delta)}}\frac{1}{\sqrt{m}},
		\end{align}
		where
		$$
		H_{n,\delta}^{\prime}:=\frac{1}{d}\sum_{0\leq l\leq m/\log^2m}\binom{m}{l}(1-\delta)^{m-l}\delta^lI(l,d)
		$$
		and
		$$
		H_{n,\delta}^{\prime\prime}:=\frac{1}{d}\sum_{m/\log^2m< l\leq m}\binom{m}{l}(1-\delta)^{m-l}\delta^lI(l,d).
		$$
		
		For $H_{n,\delta}^{\prime}$, we use the trivial bound $|\cos(x)|\leq1$ and the assumption that $\delta\leq 1/2$, then 
		\begin{align*}
				H_{n,\delta}^{\prime}\ll\sum_{0\leq l\leq m/\log^2m}\binom{m}{l}(1-\delta)^{m-l}\delta^l\ll(1-\delta)^m\sum_{0\leq l\leq m/\log^2m}\binom{m}{l},
		\end{align*}
		Notice that the binomial coefficient
		$$
		\binom{m}{l}=\frac{m!}{l!\ (m-l)!}\leq \frac{m^l}{l!},
		$$
		which yields
		$$
			H_{n,\delta}^{\prime}\ll(1-\delta)^mm^{m/\log^2m}\sum_{0\leq l\leq m/\log^2m}\frac{1}{l!}\ll(1-\delta)^mm^{m/\log^2m}.
		$$
		We infer from Lemma \ref{lem:bound for binomial paobability} that, for $m$ being sufficiently large,
		\begin{align*}
			(1-\delta)^mm^{m/\log^2m}=\big((1-\delta)e^{1/\log m}\big)^m\leq(1-\frac{1}{2}\delta)^m\ll\frac{1}{\sqrt{\frac{1}{2}\delta(1-\frac{1}{2}\delta)}}\frac{1}{\sqrt{m}}.
		\end{align*}
		 Hence we have
		\begin{align}\label{eq:H_n,1<<}
				H_{n,\delta}^{\prime}\ll\frac{1}{\sqrt{\delta(1-\delta)}}\frac{1}{\sqrt{n}}.
		\end{align}
		
		For $H_{n,\delta}^{\prime\prime}$, we use the second bound in Lemma \ref{lem:cos x<<} that
		$$		H_{n,\delta}^{\prime\prime}\ll\sum_{m/\log^2m< l\leq m}\binom{m}{l}\frac{(1-\delta)^{m-l}\delta^l}{\sqrt{l}}\ll\frac{\log m}{\sqrt m}\sum_{m/\log^2m< l\leq m}\binom{m}{l}(1-\delta)^{m-l}\delta^l\ll \frac{\log m}{\sqrt m},
		$$
		this immedicately gives
		\begin{align}\label{eq:H_n,2<<}
				H_{n,\delta}^{\prime\prime}\ll\frac{\log n}{\sqrt{n}}.
		\end{align}
		Inserting \eqref{eq:H_n,1<<} and \eqref{eq:H_n,2<<} into \eqref{eq:H_n<<H_n,1+H_n,2}, we get
		$$
		H_{n,\delta}\ll \frac{1}{\sqrt{\delta(1-\delta)}}\frac{1}{\sqrt{n}}+\frac{\log n}{\sqrt{n}}.
		$$
		Further by the estimate that $\sqrt{\delta(1-\delta)}\leq (\delta+1-\delta)/2=1/2$, we have, for even $n$,
		\begin{align}\label{eq:H_n<< for n is even}
				H_{n,\delta}\ll\frac{1}{\sqrt{\delta(1-\delta)}}\frac{\log n}{\sqrt{n}}.
		\end{align}
		
		It remains to consider the case when $n$ is odd. We employ the bound $
		|1-\delta+\delta\cos(2\pi h/d)|\leq 1$ and obtain
		$$
		H_{n,\delta}\ll\frac{1}{d}\sum_{1\leq h\leq d-1}\big(1-\delta+\delta\cos(\frac{2\pi h}{d})\big)^{\frac{n-1}{2}}=H_{n-1,\delta}.
		$$
		Since $n-1$ is even, we have
		$$
			H_{n,\delta}\ll 	H_{n-1,\delta} \ll\frac{1}{\sqrt{\delta(1-\delta)}}\frac{\log (n-1)}{\sqrt{n-1}}\ll\frac{1}{\sqrt{\delta(1-\delta)}}\frac{\log n}{\sqrt{n}}
		$$
		for odd $n$. This combines with \eqref{eq:H_n<< for n is even} gives our results.
	\end{proof}
	
	\begin{lemma}\label{lem:estimate of beta_{t_i}}
		Let $n\geq 1,\ l\geq 2$ be integers. Suppose integer $1\leq i\leq l-1$ is fixed and integers $1\leq t_1<\cdots<t_i\leq l-1$, and $0<\lambda_{t_1},\cdots,\lambda_{t_i},\eta_i<1$ with $\eta_i=1-(\lambda_{t_1}+\cdots+\lambda_{t_i})$. Then for any integer $d>1$, we have
		$$
		\frac{1}{d^{l-1}}\sum_{1\leq h_{t_1},\cdots,h_{t_i}\leq d-1}\Big|\lambda_{t_1}e\big(\frac{h_{t_1}}{d}\big)+\cdots+\lambda_{t_{i}}e\big(\frac{h_{t_i}}{d}\big)+\eta_i\Big|^n=\frac{1}{\sqrt{2\eta_i\lambda_{t_i}(1-2\eta_i\lambda_{t_i})}}O\Big(\frac{\log n}{\sqrt{n}}\Big)
		$$
		as $n\rightarrow\infty$, where $e(x):=e^{2\pi ix}$ for any real number $x$.
	\end{lemma}
	
	\begin{proof}
		We write
		$$
		J_{n,{({\lambda}_{t_a}})_{a\leq i}}=J_{n,{({\lambda}_{t_a}})_{a\leq i}}(d,l):=\frac{1}{d^{l-1}}\sum_{1\leq h_{t_1},\cdots,h_{t_i}\leq d-1}\Big|\lambda_{t_1}e\big(\frac{h_{t_1}}{d}\big)+\cdots+\lambda_{t_{i}}e\big(\frac{h_{t_i}}{d}\big)+\eta_i\Big|^n.
		$$
		By the definition of $e(x)$ and the formula $\cos(2\theta)=2\cos^2(\theta)-1$, we expand the $n$-th power inside the above sum, then
		$$
		J_{n,{({\lambda}_{t_a}})_{a\leq i}}=\frac{1}{d^{l-1}}\sum_{1\leq h_{t_1},\cdots,h_{t_i}\leq d-1}\Big(\sum_{1\leq a\leq i}\lambda^2_{t_a}+\eta^2_i-2\sum_{1\leq a<b\leq i}\lambda_{t_a}\lambda_{t_b}-2\ \eta_i\sum_{1\leq b\leq i}\lambda_{t_b}+J_{n,{({\lambda}_{t_a}})_{a\leq i}}^{\prime}\Big)^{\frac{n}{2}},
		$$
		where
		$$
		J_{n,{({\lambda}_{t_a}})_{a\leq i}}^{\prime}:=4\sum_{1\leq a<b\leq i}\lambda_{t_a}\lambda_{t_b}\cos^2\big(\frac{\pi(h_{t_a}-h_{t_b})}{d}\big)+4\eta_i\sum_{1\leq b\leq i}\lambda_{t_b}\cos^2\big(\frac{\pi h_{t_b}}{d}\big).
		$$
		Trivially, we have
		$$
		J_{n,{({\lambda}_{t_a}})_{a\leq i}}^{\prime}\leq 4\sum_{1\leq a<b\leq i}\lambda_{t_a}\lambda_{t_b}+4\eta_i\sum_{1\leq b\leq i-1}\lambda_{t_b}+4\eta_i\lambda_{t_i}\cos^2\big(\frac{\pi h_{t_i}}{d}\big).
		$$
		Then we apply $\cos(2\theta)=2\cos^2(\theta)-1$ again, there holds
		$$
	J_{n,{({\lambda}_{t_a}})_{a\leq i}}\ll\frac{1}{d^{l-1}}\sum_{1\leq h_{t_1},\cdots,h_{t_i}\leq d-1}\Big(\sum_{1\leq a\leq i}\lambda^2_{t_a}+\eta^2_i+2\sum_{1\leq a<b\leq i}\lambda_{t_a}\lambda_{t_b}+2\eta_i\sum_{1\leq b\leq i-1}\lambda_{t_b}+2\eta_i\lambda_{t_i}\cos\big(\frac{2\pi h_{t_i}}{d}\big)\Big)^{\frac{n}{2}}.
		$$
		If we set $\delta=2\eta_i\lambda_{t_i}$, then $0<\delta\leq 1/2$ and
		$$
		J_{n,{({\lambda}_{t_a}})_{a\leq i}}\ll\frac{1}{d}\sum_{1\leq h_{t_i}\leq d-1}
		\Big(1-\delta+\delta\cos\big(\frac{2\pi h_{t_i}}{d}\big)\Big)^{\frac{n}{2}}.
		$$
		It is then straightforward to complete the proof by Lemma \ref{lem:the sum in two dimensions}.
	\end{proof}
	
	Now we begin to prove Lemma \ref{lem:sum over u} by the above useful lemmas. 
	\vspace{0.3cm}
	
	\noindent\textbf{Lemma \ref{lem:sum over u}.}  \textit{Let $n\geq 1,\ k\geq 2$ be integers and $c_1,\cdots,c_{k-1}\in\mathbb{Z}$. Assume ${\boldsymbol{\alpha}}:=(\alpha_{1},\cdots,\alpha_{k-1})$ satisfying $0<\alpha_{1},\cdots,\alpha_{k-1}<1$ and $\alpha_{1}+\cdots+\alpha_{k-1}<1$. Then for any integer $d\geq 1$, we have
		$$
		\sum_{\substack{0\leq u_1,\cdots,u_k\leq n\\u_1+\cdots+u_k= n\\ u_r\equiv c_r(\bmod d),\\ \forall 1\leq r\leq k-1}}\binom{n}{u_1,\cdots,u_k}\alpha_1^{u_1}\cdots\alpha_{k-1}^{u_{k-1}}\alpha_{k}^{u_k}=\frac{1}{d^{k-1}}+\sum_{1\leq t\leq k-1}\frac{1}{\sqrt{\alpha_t\alpha_{k}}}O_{k}\Big(\frac{\log n}{\sqrt{n}}\Big)
		$$
		as $n\rightarrow\infty$, where $\alpha_{k}:=1-\sum_{r=1}^{k-1}\alpha_r$.}
	\begin{proof}
		For the sake of simplicity, denote
		$$
		L_{n,d,k,\boldsymbol{\alpha}}=L_{n,d,k,\boldsymbol{\alpha}}(c_r)_{r\leq k-1}:=\sum_{\substack{0\leq u_1,\cdots,u_k\leq n\\u_1+\cdots+u_k= n\\ u_r\equiv c_r(\bmod d),\\ \forall 1\leq r\leq k-1}}P_{n,{\bf u},{\boldsymbol{\alpha}}},
		$$	
		where $P_{n,{\bf u},{\boldsymbol{\alpha}}}$ is defined by \eqref{eq:def of P_{n,u}}.
		Appealing to the orthogonality of additive characters
		$$
		\frac{1}{d}\sum_{0\leq h\leq d-1}e\Big(\frac{hn}{d}\Big)=\left\{
		\begin{aligned}
			&1,\ \ \ \ {\rm{if}}\ d\mid n,\\
			&0,\ \ \ \ \rm{otherwise},
		\end{aligned}
		\right.
		$$	
		where $n\in \mathbb{Z}$ and $d\in\mathbb{Z}_{>0}$, we have
		\begin{align*}
			L_{n,d,k,\boldsymbol{\alpha}}=\frac{1}{d^{k-1}}\sum_{\substack{0\leq u_1,\cdots,u_k\leq n\\u_1+\cdots+u_k= n}}P_{n,{\bf u},{\boldsymbol{\alpha}}}\prod_{r=1}^{k-1}\sum_{0\leq h_r\leq d-1}e\Big(\frac{h_r(u_r-c_r)}{d}\Big).
		\end{align*}
		We change the order of the summations and apply the binomial theorem, then
		$$
		L_{n,d,k,\boldsymbol{\alpha}}=\frac{1}{d^{k-1}}\sum_{0\leq h_1,\cdots, h_{k-1}\leq d-1}\prod_{r=1}^{k-1}e(-\frac{h_rc_r}{d})\Big(\alpha_{1}e\big(\frac{h_1}{d}\big)+\cdots+\alpha_{k-1}e\big(\frac{h_{k-1}}{d}\big)+\alpha_{k}\Big)^{n},
		$$
		where $\alpha_{k}:=1-\sum_{r=1}^{k-1}\alpha_r$.
		Taking out the term $h_1=h_2=\cdots =h_{k-1}=0$ and using the fact $|e(x)|=1$, we obtain
		\begin{align}\label{eq:L=}
		L_{n,d,k,\boldsymbol{\alpha}}=\frac{1}{d^{k-1}}+O\big(L_{n,d,k,\boldsymbol{\alpha}}^{\prime}\big),
		\end{align}
		where
		$$
		L_{n,d,k,\boldsymbol{\alpha}}^{\prime}:=\frac{1}{d^{k-1}}\sum_{1\leq i\leq k-1}\sum_{\substack{1\leq m_1,\cdots,m_i\leq k-1\\m_1<\cdots<m_i}}\sum_{1\leq h_{m_1},\cdots,h_{m_i}\leq d-1}\Big|\alpha_{m_1}e\big(\frac{h_{m_1}}{d}\big)+\cdots+\alpha_{m_{i}}e\big(\frac{h_{m_i}}{d}\big)+\eta_{i}\Big|^{n}
		$$
		with $\eta_{i}:=1-(\alpha_{m_1}+\cdots+\alpha_{m_{i}})$.
		Then Lemma \ref{lem:estimate of beta_{t_i}} indicates that
		$$
		L_{n,d,k,\boldsymbol{\alpha}}^{\prime}\ll\frac{\log n}{\sqrt{n}}\sum_{1\leq i\leq k-1}\sum_{\substack{1\leq m_1,\cdots,m_i\leq k-1\\m_1<\cdots<m_i}}\frac{1}{\sqrt{2\eta_i\alpha_{m_i}(1-2\eta_i\alpha_{m_i})}},
		$$
		where $0<2\eta_i\alpha_{m_i}\leq 1/2$. It follows $\sqrt{1-2\eta_i\alpha_{m_i}}\geq 1/\sqrt{2}$ and 
		$$
		2\eta_i\alpha_{m_i}=2\alpha_{m_i}\big(1-(\alpha_{m_1}+\cdots+\alpha_{m_{i}})\big)\geq 2\big(1-(\alpha_1+\cdots+\alpha_{k-1})\big)=2\alpha_k\alpha_{m_i}
		$$
		that
		\begin{align}\label{eq:L'<<}
			\nonumber L_{n,d,k,\boldsymbol{\alpha}}^{\prime}&\ll\frac{\log n}{\sqrt{n}}\sum_{1\leq i\leq k-1}\sum_{1\leq m_i\leq k-1}\frac{1}{\sqrt{\alpha_{m_i}\alpha_k}}\sum_{1\leq m_1,\cdots, m_{i-1}\leq k-1}1\\
			&\ll_k \frac{\log n}{\sqrt{n}}\sum_{1\leq t\leq k-1}\frac{1}{\sqrt{\alpha_{t}\alpha_k}}.
		\end{align}
		Now our required result follows by \eqref{eq:L=} and \eqref{eq:L'<<}.
	\end{proof}
	
	\noindent\textbf{Acknowledgements.}
	The authors are partially supported by the National Natural Science Foundation of China (No. 12201346) and Shandong Provincial Foundation (No. 2022HWYQ-046 \& No. ZR2022QA001).  The second listed author thanks the support of the Chinese Scholarship Council and the hospitality of the Institute of Financial Mathematics and Applied Number Theory at the Johannes Kepler University  Linz.   

\end{document}